\documentclass[11pt,lefteqn]{article}

\usepackage{amsmath,amsfonts,amsthm,amssymb,graphics,epsfig,color,soul}
\usepackage{hyperref} 
\usepackage{float} 

\DeclareMathOperator*{\supess}{supess}
\newtheorem{theorem}{Theorem}

\newtheorem{corollary}[theorem]{Corollary}

\begin{document}
\date {\today}
\title{Comparison of Some Bounds on Norms of Functions of a Matrix or Operator
\thanks{
\textbf{Funding:} This material is based on work supported by the National Science Foundation under Grant No. DGE-2140004. Any opinions, findings, and conclusions or recommendations expressed in this material are those of the authors and do not necessarily reflect the views of the National Science Foundation.}}

\author {Anne Greenbaum\footnote{University of Washington,
Applied Math Dept., Box 353925, Seattle, WA 98195.  email:  greenbau@uw.edu},
Natalie Wellen\footnote{University of Washington, 
Applied Math Dept., Box 353925, Seattle, WA 98195.  email:  nwellen@uw.edu}
}


\maketitle

\begin{abstract} 
We use results in [M.~Crouzeix and A.~Greenbaum, {\em Spectral sets:  
numerical range and beyond}, SIAM Jour.~Matrix Anal.~Appl., 40 (2019), 
pp.~1087-1101] to derive upper bounds on the norm of a function $f$ of a matrix
or operator $A$ based on the infinity-norm of $f$ on various regions
in the complex plane.  We compare these results to those that can be 
derived from a straightforward application of the Cauchy integral formula
by replacing the norm of the integral by the integral of the resolvent norm.
While, in some cases, the new upper bounds on $\| f(A) \|$ are {\em much}
tighter than those from the Cauchy integral formula, we show that in many
cases of interest, the two bounds are of the same order of magnitude, with
that from the Cauchy integral formula actually being slightly smaller.
We give a partial explanation of this in terms of the numerical range of the resolvent
at points near an ill-conditioned eigenvalue.
\end{abstract}

\paragraph{2000 Mathematical subject classifications\,:}47A25 ; 47A30

\noindent{\bf Keywords\,:}{ numerical range, spectral set}

\section{Introduction}
Let $A$ be an $n$ by $n$ matrix or a bounded linear operator on a complex 
Hilbert space $(H, \langle \cdot , \cdot \rangle , \| \cdot \|)$.
A closed set $\Omega \subset \mathbb{C}$ is a $K$-spectral set for $A$ 
if the spectrum of $A$ is contained in $\Omega$ and if, for all rational 
functions $f$ bounded in $\Omega$, the following inequality holds:
\begin{equation}
\| f(A) \| \leq K \| f \|_{\Omega} , \label{Kspectral}
\end{equation}
where $\| \cdot \|$ on the left denotes the norm in $H$ and 
$\| \cdot \|_{\Omega}$ on the right denotes the $\infty$-norm on $\Omega$.  
It was shown in \cite{CP} that the closure of the numerical range, 
\begin{equation}
W(A) := \{ \langle Aq,q \rangle : q \in H,~\| q \| = 1 \} ,
\label{numericalrange}
\end{equation}
is a $(1 + \sqrt{2})$-spectral set for $A$.  This was extended in \cite{CG} 
to show that other regions in the complex plane are $K$-spectral sets.  
In particular, it was shown that the numerical
range with a circular hole or cutout is a $(3 + 2 \sqrt{3})$-spectral set.

In this paper, we use theorems proved in \cite{CG} to derive values
of $K$ for which (\ref{Kspectral}) holds for other regions $\Omega$.
A simple way to find such a $K$ value for a given region $\Omega$
containing the spectrum of $A$ in its interior is to use the Cauchy integral formula, 
replacing the norm of the integral by the integral of the resolvent norm: 
\[
f(A) = \frac{1}{2 \pi i} \int_{\partial \Omega} ( \zeta I - A )^{-1}
f( \zeta )\,d \zeta \Rightarrow
\| f(A) \| \leq \frac{1}{2 \pi} \left( \int_{\partial \Omega} 
\| ( \zeta I - A )^{-1} \|~| d \zeta | \right) \| f \|_{\Omega} .
\]
Thus one can always take
\begin{equation}
K = \frac{1}{2 \pi} \int_{\partial \Omega} \| ( \zeta I - A )^{-1} \|~
| d \zeta | . \label{KCauchy}
\end{equation}
The main goal of \cite{CG} was to produce $K$ values that are independent of $A$
for certain regions $\Omega$ (that do depend on $A$),
but it was also hoped that the values derived there would be 
smaller than those in (\ref{KCauchy}).
We will compare these $K$ values for various sets $\Omega$.  For some sets, we will also 
compare these values
to what we believe to be the optimal $K$ value.  This is computed numerically 
using an optimization code and, at least, provides a {\em lower bound} on $K$.

One way to calculate $K$ is to take $\frac{1}{2 \pi}$ times a boundary integral of the resolvent norm.
The main theorem in \cite{CG} (Theorem \ref{thm:main} in this paper), however, relates
the value of $K$ not to $\frac{1}{2 \pi}$ times a boundary integral of the resolvent norm but to 
a boundary integral of $\frac{1}{\pi}$ times the absolute value of the minimum point 
in the spectrum of the Hermitian part of a certain unit scalar times the resolvent.  
This integrand is equivalent to $\frac{1}{\pi}$ times the infimum of the real part
of the numerical range of this unit scalar times the resolvent.  If the absolute
value of this infimum turns out to be much less than the {\em numerical radius}
(the supremum of the absolute values of points in the numerical range of the resolvent,
which is between $\frac{1}{2}$ and $1$ times the norm of the resolvent), 
then Theorem \ref{thm:main} may give a much smaller $K$ value than that in (\ref{KCauchy});
on the other hand, if the absolute value of this infimum turns out to be almost
equal to the numerical radius of the resolvent, then the two $K$ values may be close, with
formula (\ref{KCauchy}) actually producing a somewhat smaller value.
We show that this latter situation holds in a number of cases of interest and we give 
a partial explanation as to why.  This observation was already hinted at in 
\cite{CGL}, where it was demonstrated numerically that the minimum point in the spectrum of
the Hermitian part of this scalar times the resolvent $( \zeta I - A )^{-1}$
tends to decrease rapidly as $\zeta$ moves to curves farther and farther inside $W(A)$.

The organization of this paper is as follows.  
In section \ref{previous results} we establish notation and review results from \cite{CG}.  
In section \ref{extensions} we extend these results and show how they
can be applied to an arbitrary region containing the spectrum of $A$ to determine a value of $K$ 
for which the region is a $K$-spectral set.
In section \ref{relation} we explain the relationship between the $K$ values
in Theorem \ref{thm:main} and those in (\ref{KCauchy}), and
in section \ref{applications} we apply the extended results to a variety of problems.
We consider block diagonal matrices and show how the numerical range can be divided into
disjoint components that constitute a $K$-spectral set for the matrix.  
We also consider relevant $K$-spectral sets for describing the behavior of continuous and
discrete time dynamical systems.  
In section \ref{comparisons} we give concluding remarks.

\section{Results from \cite{CG}} \label{previous results}

\subsection{Notation}
Let $f$ be a rational function bounded in a closed set $\Omega$ containing the spectrum of $A$.  Assume that the boundary $\partial \Omega$ is
rectifiable and has a finite number of connected components.  From the Cauchy integral formula, 
we can write
\[
f(z) = \frac{1}{2 \pi i} \int_{\partial \Omega} \frac{f( \zeta )}{\zeta - z}\,d\zeta ,~~
f(A) = \frac{1}{2 \pi i} \int_{\partial \Omega} ( \zeta I - A )^{-1} f( \zeta )\,d\zeta .
\]
Letting $s$ denote arc length, going in a counter-clockwise direction along $\partial \Omega$,
and letting $\partial \omega \subset \mathbb{R}$ denote the values of $s$ as $\zeta (s)$ traverses 
$\partial \Omega$, the above equations can
be written in the form
\[
f(z) = \frac{1}{2 \pi i} \int_{\partial \omega} \frac{f( \zeta (s) )}{\zeta (s) - z} \zeta' (s)\,ds ,~~
f(A) = \frac{1}{2 \pi i} \int_{\partial \omega} ( \zeta (s) I - A )^{-1} f( \zeta (s) ) \zeta' (s)\,ds .
\]
We will also use the Cauchy transform of the complex conjugate $\bar{f}$:
\[
g(z) := C( \overline{f},z) := \frac{1}{2 \pi i} \int_{\partial \omega} \frac{\overline{f( \zeta (s))}}{\zeta (s) - z} \zeta' (s)\,ds ,~~
g(A) := \frac{1}{2 \pi i} \int_{\partial \omega} ( \zeta (s) I - A )^{-1} \overline{f( \zeta (s))} \zeta' (s)\,ds .
\]
Finally we define the transform of $f$ by the double layer potential kernel,
\begin{equation}
\mu ( \zeta (s),z ) := \frac{1}{\pi} \frac{d}{ds} ( \arg ( \zeta (s) - z ) ) = 
\frac{1}{2 \pi i} \left( \frac{ \zeta' (s)}{\zeta (s) - z} - 
\frac{ \overline{\zeta' (s)}}{\overline{\zeta (s)} - \bar{z}} \right) ,
\label{mu_defn}
\end{equation}
\begin{equation}
\mu ( \zeta (s),A ) = \frac{1}{2 \pi i} \left( ( \zeta (s) I - A )^{-1} \zeta' (s) - 
( \overline{\zeta (s)} I - A^{*} )^{-1} \overline{\zeta' (s)} \right) .
\label{muA_def}
\end{equation}
With these definitions, we can write
\[
S(f,z) := f(z) + \overline{g(z)} = \int_{\partial \omega} f( \zeta (s)) \mu ( \zeta (s),z)\,ds ,
\]
\[
S(f,A) := f(A) + g(A )^{*} = \int_{\partial \omega} f( \zeta (s)) \mu ( \zeta (s),A)\,ds .
\]
Further, note that $S(1,A) = 2I$ since
\[
\int_{\partial \omega} \mu ( \zeta (s),A)\,ds = \frac{1}{2 \pi i} \int_{\partial \omega} 
( \zeta (s) I - A )^{-1} \zeta' (s)\,ds + \left( \frac{1}{2 \pi i}
\int_{\partial \omega} ( \zeta (s) I - A )^{-1} \zeta' (s)\,ds \right)^{*} = I + I^{*} = 2I .
\]

\subsection{Main Results from \cite{CG}}
Define
\[
c_1 := \sup \{ \max_{z \in \Omega } | C( \bar{f}, z ) | : f \mbox{ a rational function},
\| f \|_{\Omega} \leq 1 \} .
\]
It is shown in \cite[Lemma 1]{CG} that $c_1$ satisfies
\begin{equation}
c_1 \leq \supess_{\zeta_0 \in \partial \Omega} \int_{\partial \omega} | \mu ( \zeta (s), \zeta_0 ) |\,ds .
\label{c1_bound}
\end{equation}
Define 
\begin{equation}
c_2 := \frac{1}{2} \sup \{ \| S(f,A) \| : f \mbox{ a rational function}, 
\| f \|_{\Omega} \leq 1 \} . \label{c2_def}
\end{equation}

Following is (a part of) the main theorem of \cite[Theorem 2]{CG}:

\begin{theorem}
\label{thm:main}
With $c_1$ and $c_2$ as defined above, $\Omega$ is a $K$-spectral set for $A$, where
\[
K = c_2 + \sqrt{ c_2^2 + c_1 } .
\]
\end{theorem}
 
One can use (\ref{c1_bound}) and definition (\ref{mu_defn}) to bound $c_1$ in the theorem.
If we fix $\zeta_0 \in \partial \Omega$ and let $\zeta(s)$ move around 
a curve $\Gamma_j$ that is all or part of $\partial \Omega$ then,
from the definition in (\ref{mu_defn}), $\int_{s: \zeta (s) \in \Gamma_j} | \mu ( \zeta (s),
\zeta_0 ) |\,ds$ is equal to $\frac{1}{\pi}$ times the total variation in the argument of $\zeta (s) - \zeta_0$.
For example, if $\partial \Omega$ is a circle or the boundary of a convex set such as in Figure \ref{fig:regions}(a),
then the argument of $\zeta (s) - \zeta_0$ changes by $\pi$ as
$\zeta (s)$ traverses the curve $\partial\Omega$ so that $\int_{\partial \omega} | \mu ( \zeta (s), \zeta_0 ) |\,ds = 1$.  If $\zeta_0$ lies inside a circle
or the boundary curve of a convex set such as in Figure \ref{fig:regions}(b), then the integral of $| \mu ( \zeta (s), \zeta_0 ) |$ over that piece of 
the boundary is $2$.
If $\zeta_0$ lies outside a circle of radius $r$ such as in Figure \ref{fig:regions}(c), then, if $R$
is the distance from $\zeta_0$ to the center of the circle, the argument of $\zeta (s) - \zeta_0$
goes from its initial value, say, $0$ to $\arcsin (r/R)$ to $0$, to $- \arcsin (r/R)$, and back to $0$,
for a total change of $4 \arcsin (r/R)$.  Note that for any region $\Omega$, the upper bound (\ref{c1_bound}) on $c_1$ can be computed numerically, by testing many points $\zeta_0 \in \partial \Omega$
and finding the one that leads to the largest total variation in the 
argument of $\zeta (s) - \zeta_0$, as $\zeta (s)$ traverses $\partial \Omega$.

\begin{figure}[ht]
\centerline{\epsfig{file=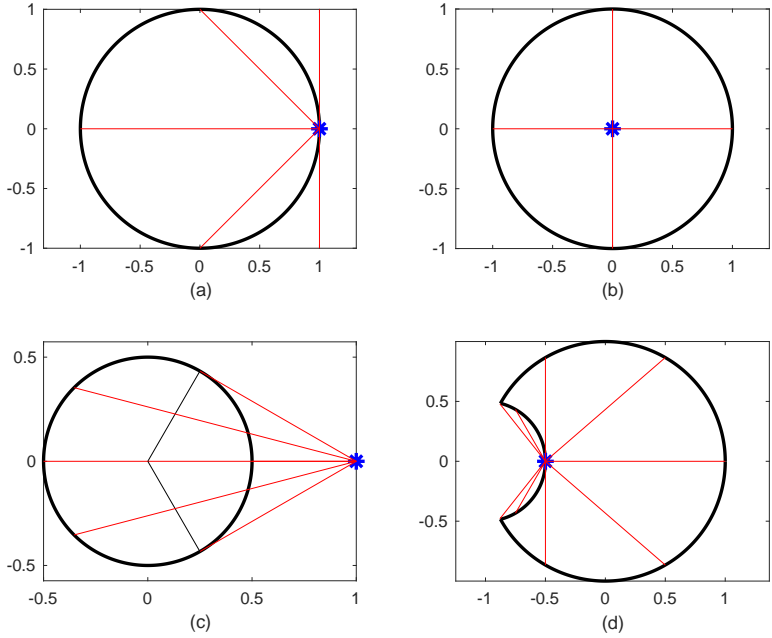,width=4in}}
\caption{Various boundary configurations.  The blue asterisk represents $\zeta_0$,
and the red lines show how the angle of the vector $\zeta (s) - \zeta_0$ changes
as $\zeta (s)$ traverses the boundary curve.}
\label{fig:regions}
\end{figure}
\medskip

To obtain upper bounds on $c_2$, we first note that if $\mu ( \zeta (s) , A )$ is positive
semidefinite (PSD) for $s \in [ s_{min} , s_{max} ]$, then
\begin{equation}
\left\| \int_{s_{min}}^{s_{max}} f( \zeta (s) ) \mu ( \zeta (s), A )\,ds \right\| \leq
\max_{s \in [s_{min} , s_{max}]} | f( \zeta (s)) |~\left\| \int_{s_{min}}^{s_{max}} 
\mu ( \zeta (s), A )\,ds \right\| .
\label{PSDresult}
\end{equation}
A proof can be obtained by noting that
\[
\left\| \int_{s_{min}}^{s_{max}} f( \zeta (s)) \mu ( \zeta (s), A )\,ds \right\| =
\sup_{\| x \| = \| y \| = 1} \left| \int_{s_{min}}^{s_{max}} f( \zeta (s) )
\left\langle \mu ( \zeta (s) ,A )y, x \right\rangle \,ds \right| ,
\]
and following the arguments in \cite[Lemma 2.3]{CP}. 
Thus if $\mu ( \zeta , A )$ is PSD for all $\zeta \in \partial \Omega$, then
$c_2 \leq 1$, since for any rational function $f$ with $\| f \|_{\Omega} \leq 1$,
\[
\| S(f,A) \|  \leq \left\| \int_{\partial \omega} \mu ( \zeta (s), A )\,ds \right\| = 
\| 2I \| = 2 ,
\]
and from definition (\ref{c2_def}), $c_2$ is bounded by half this value.
For $\Omega$ a convex set containing $W(A)$, Theorem \ref{thm:main} yields the 
Crouzeix-Palencia result \cite{CP} that $\Omega$ is a $(1 + \sqrt{2} )$-spectral set 
for $A$, since in this case $c_1 \leq 1$ and $c_2 \leq 1$.

When $\mu ( \zeta (s) , A )$ is not PSD, we will add a multiple of the identity to
$\mu ( \zeta (s) , A)$ to obtain a PSD operator.  For this, we need bounds on the minimum
value in the spectrum of $\mu ( \zeta (s) , A )$:
\begin{equation}
\lambda_{min}( \mu ( \zeta (s), A)) := \min \{ \lambda : \lambda \in 
\mbox{Sp} ( \mu ( \zeta (s) , A ) ) \} .
\label{lambdamin_def}
\end{equation}
Let $\zeta_0 = \zeta ( s_0 )$ denote a point on $\partial \Omega$ where the unit 
tangent $\zeta_0' := \left. \frac{d \zeta}{ds} \right|_{s_0}$ exists.  
Since $\mu ( \zeta (s),A)$ depends on $\zeta' (s)$, when we fix a point $\zeta_0$,
we will write $\mu ( \zeta_0 , \zeta_0' , A)$ to make this dependence clear.
Note that the half-plane $\Pi_0 :=
\{ z \in \mathbb{C} : \mbox{Im}( \zeta_0' (\overline{\zeta_0} - \bar{z})) \geq 0 \}$ 
has the same outward normal as $\Omega$ at $\zeta_0$.
The following theorem is from \cite[Lemmas 5, 7, and 8]{CG}.
For a disk about a point $\xi$ of radius $r$, the
assumption $\zeta_0 - \xi = i r \zeta_0'$ in the theorem means that 
$\partial \Omega$ and the boundary of the disk are tangent at $\zeta_0$ and 
the outward normal to $\Omega$, $\zeta_0' / i$, is the same as the inward normal to the disk.

\begin{theorem}
\label{thm:lambdamin}
If $W(A) \subset \Pi_0$,
then $\lambda_{min} ( \mu ( \zeta_0 , \zeta_0' , A )) \geq 0$,
with equality if $\zeta_0 \in \partial W(A)$.
If, for some $\xi \in \mathbb{C} \backslash \mbox{Sp} (A)$, 
$\zeta_0 - \xi = i r_1 \zeta_0'$, where $r_1 \leq 1/ \| (A - \xi I )^{-1} \|$,
then $\lambda_{min} ( \mu ( \zeta_0 ,  \zeta_0' , A )) \geq - \frac{1}{2 \pi r_1}$.
If $\zeta_0 - \xi = i r_2 \zeta_0'$, where $r_2 \leq 1/w((A - \xi I )^{-1} )$, 
then $\lambda_{min} ( \mu ( \zeta_0 , \zeta_0',  A ) ) \geq - \frac{1}{\pi r_2}$.
\end{theorem}

Note that the interior of the disks 
$\{ z \in \mathbb{C} : | z - \xi | < 1 / \| ( A - \xi I )^{-1} \| \}$
and $\{ z \in \mathbb{C} : | z - \xi | < 1/ w( (A - \xi I )^{-1} ) \}$ 
alluded to in the theorem contain no points in the spectrum of $A$ since 
$\| (A - \xi I )^{-1} \| \geq w( ( A - \xi I )^{-1} ) \geq | ( \lambda - \xi )^{-1} |$
for all $\lambda \in \mbox{Sp} (A)$; that is, the inverses of 
these quantities, which are the radii of the disks, are less than or equal to 
$| \lambda - \xi |$.

Theorems \ref{thm:main} and \ref{thm:lambdamin} can be used together
to obtain $K$ values for certain types of sets, such as the numerical 
range with a circular hole or cutout.  In the next subsection, we include
such an example from \cite{CG}.  In the following section we extend this
example in several ways and also indicate how Theorem \ref{thm:main} can
be used directly to determine a $K$ value for any set $\Omega$ containing
the spectrum of $A$.

\subsection{Example from \cite{CG}}\label{example}
Using these results, it is shown in \cite{CG} that if $\Omega = \Omega_0 \backslash
{\cal D} ( \xi , r )$, where $\Omega_0$ is a convex domain containing
$\mbox{cl}(W(A))$ (where $\mbox{cl}( \cdot )$ denotes the closure)
and ${\cal D} ( \xi , r )$ is the disk about a point
$\xi \in \mathbb{C} \backslash \mbox{Sp}(A)$ of radius $r$, where $r \leq 
1/w( (A - \xi I )^{-1})$, then $\Omega$ is a 
$(3 + 2 \sqrt{3} )$-spectral set for $A$.  This assumes that either
$\partial {\cal D} ( \xi , r ) \subset \Omega_0$ or the number of intersection 
points of $\partial \Omega_0$ and $\partial {\cal D}( \xi , r )$ is finite.

To bound $c_1$ in this case, suppose first that $\partial {\cal D} ( \xi , r ) 
\subset \Omega_0$.  
If $\zeta_0 \in \partial \Omega_0$, then as $\zeta (s)$ traverses
$\partial \Omega_0$, the argument of $\zeta (s) - \zeta_0$ changes by $\pi$,
as illustrated in Figure \ref{fig:regions}(a).  As
$\zeta (s)$ traverses $\partial {\cal D} ( \xi , r )$, the argument of $\zeta (s) -
\zeta_0$ changes by $4 \arcsin ( r/ | \zeta_0 - \xi | ) < 2 \pi$,
as illustrated in Figure \ref{fig:regions}(c).
Thus, in this case,
\[
\int_{\{ s : \zeta (s) \in \partial \Omega_0 \}} | \mu ( \zeta (s), \zeta_0 ) |\,ds = 1 ,~~
\int_{\{ s: \zeta (s) \in \partial {\cal D} ( \xi , r ) \}} 
| \mu ( \zeta (s), \zeta_0 ) |\,ds < 2 .
\]
[To simplify notation, throughout the rest of the paper we will write simply $\int_{\partial \Omega_j} \ldots\,ds$ 
in place of $\int_{\{ s: \zeta (s) \in \partial \Omega_j \}} \ldots\,ds$.]
Now suppose $\zeta_0 \in \partial {\cal D} ( \xi , r )$.  Then as $\zeta (s)$ traverses
$\partial \Omega_0$, the argument of $\zeta (s) - \zeta_0$ changes by $2 \pi$,
as illustrated in Figure \ref{fig:regions}(b), while
as $\zeta (s)$ traverses $\partial {\cal D}( \xi , r )$, the argument of $\zeta (s) -
\zeta_0$ changes by $\pi$, as illustrated in Figure
\ref{fig:regions}(a). Thus, in this case, we have
\[
\int_{\partial \Omega_0} | \mu ( \zeta (s), \zeta_0 ) |\,ds = 2 ,~~
\int_{\partial {\cal D} ( \xi , r )} | \mu ( \zeta (s), \zeta_0 ) |\,ds = 1 .
\]
It follows that for $\zeta_0$ anywhere on the boundary of $\Omega$, the
change in argument of $\zeta (s) - \zeta_0$ as $\zeta (s)$ traverses
$\partial \Omega$ is at most $3 \pi$; that is, $c_1 \leq 3$.
If, instead, the disk ${\cal D} ( \xi , r )$ intersects $\partial \Omega_0$
as in Figure \ref{fig:regions}(d), then it is clear that the total variation
in the argument of $\zeta (s) - \zeta_0$ as $\zeta(s)$ traverses $\partial \Omega$
is smaller and thus $c_1$ is again bounded by $3$.

To bound $c_2$, let $\Gamma_0 = \partial \Omega_0 \backslash 
\mbox{cl}( {\cal D} ( \xi , r ))$ and let 
$\Gamma_1 = \partial {\cal D} ( \xi , r ) \cap \mbox{cl}( \Omega_0 )$,
so that $\partial \Omega = \Gamma_0 \cup \Gamma_1$. 
Let $f$ be a function with $\| f \|_{\Omega} \leq 1$ and write 
$S(f,A) = S_0 + S_1 + S_2$, where
\[
S_0 = \int_{\Gamma_0} f( \zeta (s)) \mu ( \zeta (s),A)\,ds,~~
S_1 = \int_{\Gamma_1} f( \zeta (s)) \left( \mu ( \zeta (s),A ) + \frac{1}{\pi r} I
\right)\,ds ,~~
S_2 = - \frac{1}{\pi r} \int_{\Gamma_1} f( \zeta (s) ) I\,ds .
\]
It follows from Theorem \ref{thm:lambdamin} that for $\zeta \in \partial \Omega_0$,
$\mu ( \zeta , A )$ is PSD.  Since adding PSD operators to a PSD 
operator does not decrease the norm, we can extend the integral over $\Gamma_0$
to an integral over the entire boundary $\partial \Omega_0$ to obtain:
\[
\| S_0 \| \leq \left\| \int_{\partial \Omega_0} \mu ( \zeta (s),A)\,ds
\right\| = \| 2I \| = 2 .
\]
If $\zeta \in \partial {\cal D} ( \xi , r)$,
since $r \leq 1/w((A- \xi I )^{-1})$, Theorem \ref{thm:lambdamin} shows that 
$\mu ( \zeta ,A) + \frac{1}{\pi r} I$ is PSD, and hence
\[
\| S_1 \| \leq \left\| \int_{\Gamma_1} \left( \mu ( \zeta (s),A) +
\frac{1}{\pi r} I \right) \,ds \right\| \leq
\left\| \int_{\partial {\cal D} ( \xi , r )} \left( \mu ( \zeta (s),A) + 
\frac{1}{\pi r} I \right)\,ds \right\| 
= \frac{1}{\pi r} \int_{\partial {\cal D} ( \xi , r )} ds = 2. 
\]
Here we have used the fact that the spectrum of $A$ lies outside ${\cal D} ( \xi , r)$
and hence $\int_{\partial {\cal D} ( \xi , r )} \mu ( \zeta (s),A)\,ds = 0$.  It is clear
that $\| S_2 \| \leq 2$, since the length of $\Gamma_1$ is less than or equal to
the length of $\partial {\cal D} ( \xi , r )$, which is 
$2 \pi r$.  Thus $\| S(f,A) \| \leq 6$ and $c_2 \leq 3$.
Applying Theorem \ref{thm:main} with $c_1 = c_2 = 3$, yields the result from \cite{CG} 
that $\Omega$ is a $(3 + 2 \sqrt{3})$-spectral set for $A$.  

%

\section{Some Extensions} \label{extensions}
The arguments in section \ref{example} can be extended in some simple ways.

Suppose, for example, that $\Omega = \Omega_0 \backslash {\cal D} ( \xi , r )$
where $\Omega_0$ and ${\cal D} ( \xi , r )$ are as in section \ref{example},
but where the intersection of $\Omega_0$ and ${\cal D}( \xi , r )$ is at most
a half-disk, as pictured in Figure \ref{fig:regions}(d).
The greatest variation in the argument of $\zeta_0 - \zeta (s)$ can be attained when
$\zeta_0$ is in the position of the asterisk in the figure.
Then the total variation of the argument of
$\zeta (s) - \zeta_0$ could change by as much as $\pi / 2$ as 
$\zeta (s)$ traverses $\Gamma_1$.  It changes by the same amount as 
$\zeta (s)$ moves along $\Gamma_0$ to the point where the argument of
$\zeta (s) - \zeta_0$ matches $\zeta_0'$ or $- \zeta_0'$,
with a change of $\pi$ in between. 
The total change could therefore be as large as $2 \pi$.
It follows that in this case, for any $\zeta_0$ on $\partial \Omega$,
\[
\int_{\partial \Omega_0} | \mu ( \zeta (s), \zeta_0 ) |\,ds \leq 2 ,
\]
and therefore $c_1 \leq 2$ when at most a half-disk is removed from $\Omega_0$.
Using the same definitions of $S_0$, $S_1$, and $S_2$ as in section \ref{example},
we now observe that the length of $\Gamma_1$ is at most $\pi r$ instead of $2 \pi r$, 
so that $\| S_2 \| \leq 1$, leading to the estimate $\| S(f,A) \| \leq 5$ and 
$c_2 \leq 5/2$.  Using these values of $c_1$ and $c_2$ in Theorem \ref{thm:main}
leads to the result that $\Omega$ is a $( 2.5 + \sqrt{8.25} )$-spectral set for $A$.

If the radius $r$ of the disk removed from $\Omega_0$ satisfies
$r \leq 1/ \| ( A - \xi I )^{-1} \|$,
then from Theorem \ref{thm:lambdamin}, it follows that 
$\lambda_{min} ( \mu ( \zeta_0 , A )) \geq - \frac{1}{2 \pi r}$.
In this case, we can replace $S(f,A) = S_0 + S_1 + S_2$ by 
$S(f,A) = S_0 + \tilde{S}_1 + \tilde{S}_2$, where
\[
\tilde{S}_1 = \int_{\Gamma_1} f( \zeta (s)) \left( \mu ( \zeta (s),A )  + 
\frac{1}{2 \pi r} I \right)\,ds ,~~
\tilde{S}_2 = - \frac{1}{2 \pi r} \int_{\Gamma_1} f( \zeta (s)) I\,ds .
\]
Now
\[
\| \tilde{S}_1 \| \leq \left\| \int_{\Gamma_1} \left( \mu ( \zeta(s),A ) + 
\frac{1}{2 \pi r} I \right)\,ds \right\| \leq
\frac{1}{2 \pi r} \int_{\partial {\cal D} ( \xi , r)} ds = 1 ,
\]
and $\| \tilde{S}_2 \| \leq 1$.  With $c_1 = 3$ and $c_2 = 2$,
it follows from Theorem \ref{thm:main} that $\Omega$ is a $(2 + \sqrt{7})$-spectral set,
and if the intersection of $\Omega_0$ and ${\cal D} ( \xi , r )$ is at most
a half-disk, then with $c_1 = 2$, and $\| \tilde{S}_2 \| \leq 1/2$, 
we can take $c_2 = 7/4$, and then it follows from Theorem \ref{thm:main} that
this is a $4$-spectral set for $A$.

\subsection{Removing More Disks}  \label{disks}
The techniques of section \ref{example} can be used to bound $K$ when multiple disks
are removed from $\Omega_0 \supset \mbox{cl}(W(A))$.   
%

\begin{corollary}
\label{cor:mdisks}
Suppose $\Omega_0 \supset \mbox{cl}(W(A))$ and $\Omega$ is obtained from $\Omega_0$
by removing $m$ disks centered at points $\xi_1 , \ldots , \xi_m$, with the radius $r_j$
of disk $j$ equal to either $1/ \| (A - \xi_j I )^{-1} \|$ or $1/ w(( A - \xi_j I )^{-1} )$.
Set $p_j = 1$ if $r_j = 1/ \| (A - \xi_j I )^{-1} \|$ and $p_j = 2$ if
$r_j = 1/w(( A - \xi_j I )^{-1} )$.  Then $\Omega$ is a $K$-spectral set for $A$ with
\begin{equation}
K \leq \left( 1 + \sum_{j=1}^m p_j \right) + \sqrt{ \left( 1 + \sum_{j=1}^m p_j \right)^2 + 2m + 1} .
\label{mdisks}
\end{equation}
\end{corollary}

\begin{proof}
Consider first the simplest case, where the disks ${\cal D}_1 ( \xi_1 , r_1 ), \ldots ,
{\cal D}_m ( \xi_m , r_m )$ do not overlap and lie entirely inside $\Omega_0$.
For $\zeta_0 \in \partial \Omega_0$, the total variation in $\arg ( \zeta (s) - \zeta_0 )$
becomes
\[
\pi + 4 \sum_{j=1}^{m} \arcsin \left( \frac{1}{r_j | \zeta_0 - \xi_j |} \right)
\leq \pi + 2 m \pi .
\]
If $\zeta_0$ lies on $\partial {\cal D}_i$, then the change in $\arg ( \zeta (s) -
\zeta_0 )$ is $2 \pi$ as $\zeta (s)$ traverses $\partial \Omega_0$ and $\pi$
as $\zeta (s)$ traverses $\partial {\cal D}_i$.  The total change is
\[
3 \pi + 4 \sum_{\stackrel{j=1}{j \neq i}}^{m} \arcsin \left( \frac{1}{r_j | \zeta_0
- \xi_j |} \right) \leq 3 \pi + 2(m-1) \pi .
\]
In either case, the total variation of $\arg ( \zeta (s) - \zeta_0 )$ is at most
$(2m+1) \pi$, so that $c_1 \leq 2m+1$.

To bound $c_2$, write $S(f,A) = S_0 + \sum_{j=1}^m S_j + \sum_{j=1}^m S_{m+j}$, where
\[
S_0 = \int_{\partial \Omega_0} f( \zeta (s)) \mu ( \zeta(s), A)\,ds ,~~
S_j = \int_{\partial {\cal D}_j} f( \zeta (s)) \left( \mu ( \zeta (s),A) ) + 
\frac{p_j}{ 2 \pi r_j} I \right)\,ds ,
\]
\[
S_{m+j} = - \frac{p_j}{ 2 \pi r_j} \int_{\partial {\cal D}_j} f( \zeta (s)) I\,ds ,~~ 
j=1, \ldots , m .
\]
Then
\[
\| S_0 \| \leq 2 ,~~ \| S_j \| \leq p_j ,~~\| S_{m+j} \| \leq p_j ,~~
j=1, \ldots , m . 
\]
It follows that 
\[
\| S(f,A) \| \leq 2 + 2 \sum_{j=1}^m p_j ,
\]
and $c_2 \leq 1 + \sum_{j=1}^m p_j$.  Applying Theorem \ref{thm:main} with $c_1 = 2m+1$ and 
$c_2 = 1 + \sum_{j=1}^m p_j$, we arrive at (\ref{mdisks}).  This upper bound holds for
other configurations as well, where $c_1$ and/or $c_2$ may be smaller because disks overlap
or only partially intersect with $\Omega_0$.
\end{proof}

Note that when the disks in Corollary \ref{cor:mdisks} overlap or only partially intersect
with $\Omega_0$, better bounds on $K$ may be attainable by considering each geometry individually.

\subsection{Other $K$-Spectral Sets} \label{OtherKSpectral}
In the previous subsection, we made use of Theorem \ref{thm:lambdamin} to derive
values of $K$ that are independent of the operator $A$ for special types of regions $\Omega$
(that {\em do} depend on $A$).  For a given operator $A$ and region $\Omega$ containing
the spectrum of $A$,
one can use Theorem \ref{thm:main} directly to derive $K$ values (that depend on both $A$ and $\Omega$),
but in most cases, these values will have to be computed numerically.
A bound on the parameter $c_1$ depends only on the geometry of $\Omega$,
while $c_2$ can be bounded using computed values of $\lambda_{min} ( \mu ( \zeta (s) , A))$.

Examples of regions $\Omega$ that might be of interest include the intersection
of $W(A)$ with the left half-plane, when the spectrum of $A$ lies in the
left half-plane but $W(A)$ extends into the right half-plane, or the intersection
of $W(A)$ with the unit disk, when the spectrum of $A$ lies inside the unit disk.
In the first case, if it can be shown that the intersection of $W(A)$ with
the left half-plane is a $K$-spectral set for $A$, then $K$ is an upper bound
on the amount by which the norm of the solution to the continuous time dynamical system
$y' (t) = A y(t)$, $t > 0$, can grow over its initial value before eventually
decaying to $0$.  In the second case, if it can be shown that the intersection
of $W(A)$ with the unit disk is a $K$-spectral set for $A$, then $K$ is an
upper bound on the amount by which the norms of powers of $A$, $\| A^j \|$, $j=0,1, \ldots$
can grow.

In either of these cases, the set $\Omega = W(A) \cap \mbox{(left half-plane)}$ or
$\Omega = W(A) \cap \mbox{(unit disk)}$ is convex, so $c_1 = 1$.  To bound $c_2$,
let $\Gamma_0$ denote the part of $\partial W(A)$ that is retained as part of 
$\partial \Omega$ and let $\Gamma_1$
denote the line segment or circular arc resulting from the intersection of $W(A)$
with the imaginary axis or the unit circle.  
Then $\partial \Omega = \Gamma_0 \cup \Gamma_1$.
For $f \in {\cal A} ( \Omega )$ with $\| f \|_{\Omega} \leq 1$, define
\[
S_0 = \int_{\Gamma_0} f( \zeta (s)) \mu ( \zeta (s), A)\,ds ,~~
S_1 = \int_{\Gamma_1} f( \zeta (s)) ( \mu ( \zeta (s),A) + \gamma (s) I )\,ds ,~~
S_2 = - \int_{\Gamma_1} f( \zeta (s)) \gamma (s) I\,ds ,
\]
where $\gamma (s) \geq - \lambda_{min} ( \mu ( \zeta (s), A))$.
Proceeding as in section \ref{example}, since $\mu ( \zeta (s) , A )$ is PSD
for $\zeta (s) \in \partial W(A)$, we can write
\[
\| S_0 \| \leq \left\| \int_{\Gamma_0} \mu ( \zeta (s),A)\,ds \right\|
\leq \left\| \int_{\partial W(A)} \mu ( \zeta (s),A )\,ds \right\| =
\| 2I \| = 2 .
\]
Similarly, since $\mu ( \zeta (s),A) + \gamma (s) I$ is PSD
on $\Gamma_1$ and $\mu ( \zeta (s),A)$ is PSD
on $\partial W(A)$, if we let $\Gamma_2$ denote the part of
$\partial W(A)$ that was discarded and define $\gamma (s)$ to be $0$
on $\Gamma_2$, then we have
\[
\| S_1 \| \leq \left\| \int_{\Gamma_1} ( \mu ( \zeta (s),A) + \gamma (s) I )\,ds \right\| \leq
\left\| \int_{\Gamma_1 \cup \Gamma_2} ( \mu ( \zeta (s),A) + \gamma (s) I )\,ds
\right\| = \left| \int_{\Gamma_1 \cup \Gamma_2} \gamma (s) \right| = \int_{\Gamma_1} | \gamma (s) |\,ds .
\]
Finally, we can write
\[
\| S_2 \| \leq \int_{\Gamma_1} | \gamma (s) |\,ds .
\]
Since $S(f,A) = S_0 + S_1 + S_2$, it follows that $\| S(f,A) \| \leq 2 + 2 \int_{\Gamma_1} | \gamma (s) |\,ds$
and therefore
\begin{equation}
c_2 \leq 1 + \int_{\Gamma_1} | \gamma(s) |\,ds . \label{c2formula}
\end{equation}

In general, suppose a set $\Omega$ consists of $m$ disjoint, simply connected 
regions $\Omega_1 , \ldots , \Omega_m$ with boundaries $\Gamma_1 , \ldots , \Gamma_m$.  
An example might be the $\epsilon$-pseudospectrum of $A$:
\[
\Lambda_{\epsilon} (A) := \{ z \in \mathbb{C} : \| (zI-A )^{-1} \| > \epsilon^{-1} \}
\]
For this set, the value (\ref{KCauchy}) is easy to compute:
\[
K = \frac{{\cal L}( \partial \Lambda_{\epsilon} )}{2 \pi \epsilon} ,
\]
where ${\cal L} ( \cdot )$ denotes the length of the curve.  In this case, it may
be difficult to come up with
an analytic expression for the bound (\ref{c1_bound}) on $c_1$.  This bound can be
estimated numerically (to any desired accuracy), however, by first discretizing $\partial
\Lambda_{\epsilon} (A)$, then considering each discretization point as a possible value
for $\zeta_0$ in (\ref{c1_bound}), determining the total variation of the argument
of $\zeta (s) - \zeta_0$ as $\zeta (s)$ traverses the discretized $\partial \Lambda_{\epsilon} (A)$,
and finally taking $c_1$ to be $\frac{1}{\pi}$ times the maximum value of this total variation.
To compute a bound on $c_2$, 
let $f$ be  any rational function with $\| f \|_{\Lambda_{\epsilon} (A)} 
\leq 1$, and write $S(f,A) = S_1 + S_2$, where 
\[
S_1 = \int_{\cup_j \Gamma_j} f( \zeta (s) ) ( \mu ( \zeta (s),A ) + \gamma (s) I )\,ds ,~~
S_2 = - \int_{\cup_j \Gamma_j} f( \zeta (s)) \gamma (s) I\,ds .
\]
Taking $\gamma (s)$ to be greater than or equal to $- \lambda_{min} ( \mu ( \zeta (s), A) )$, so that
$\mu ( \zeta (s),A) + \gamma (s) I$ is PSD, we can write
\[
\| S_1 \| \leq  \left\| \int_{\cup_j \Gamma_j} ( \mu ( \zeta (s),A) + \gamma (s) I )\,ds \right\| \leq
2 + \left\| \int_{\cup_j \Gamma_j} \gamma (s) I\,ds \right\| \leq 2 + \int_{\cup_j \Gamma_j} | \gamma (s) |\,ds ,
\]
and similarly,
\[
\| S_2 \| \leq \int_{\cup_j \Gamma_j} | \gamma (s) |\,ds .
\]
In this case, $\| S(f,A) \| \leq 2 + 2 \int_{\cup_j \Gamma_j} | \gamma (s) |\,ds$ and 
therefore
\[
c_2 \leq 1 + \int_{\cup_j \Gamma_j} | \gamma (s) |\,ds .
\]

\section{Relation between $K$ Values from Theorem \ref{thm:main} and from 
(\ref{KCauchy})} \label{relation}

Recall the definition of $\mu ( \zeta (s), A )$ in (\ref{muA_def}), which we also
write as $\mu ( \zeta_0 , \zeta_0' , A)$ if $\zeta (s) = \zeta ( s_0 ) = \zeta_0$
and $\zeta_0' = \left. \frac{d \zeta}{ds} \right|_{s_0}$.  Since the magnitude
of $\zeta_0'$ is $1$, it can be written in the form $e^{i \theta_0}$
for some $\theta_0 \in [0, 2 \pi )$.  Therefore, using definition (\ref{muA_def}), 
we can write
\begin{equation}
\mu ( \zeta_0 , \zeta_0' , A ) = \frac{1}{2 \pi} \left[ e^{i ( \theta_0 - \pi/2)}
( \zeta_0 I - A )^{-1}  +  e^{-i ( \theta_0 - \pi /2)} \left( ( \zeta_0 I - A )^{-1}
\right)^{*} \right] . \label{mu_expression}
\end{equation}
It follows that $\lambda_{min} ( \mu ( \zeta_0 , \zeta_0' , A ) )$ is $\frac{1}{\pi}$
times the minimum point in the spectrum of the Hermitian part of 
$e^{i ( \theta_0 - \pi /2 )} ( \zeta_0 I - A )^{-1}$, which is $\frac{1}{\pi}$ 
times the smallest real part of points in
$\mbox{cl} ( W ( e^{i ( \theta_0 - \pi / 2 )} ( \zeta_0 I - A )^{-1} ))$.
We conclude that 
$| \lambda_{min} ( \mu ( \zeta_0 , \zeta_0', A )) |$ is less than or equal to
$\frac{1}{\pi}$ times
the numerical radius of $e^{i ( \theta_0 - \pi / 2 )} ( \zeta_0 I - A )^{-1}$,
which is the same as $\frac{1}{\pi}$ times the numerical radius of the resolvent 
$( \zeta_0 I - A )^{-1}$.

In some cases, $| \lambda_{min} ( \mu ( \zeta_0 , \zeta_0' , A )) |$ may be 
{\em much} less than $\frac{1}{\pi}$ times the numerical radius of the resolvent; 
e.g., when $\zeta_0$
lies on $\partial W(A)$ so that $\lambda_{min} ( \mu ( \zeta_0 , \zeta_0', A )) = 0$.
In these cases, one can expect a {\em much} smaller value of $K$ in Theorem \ref{thm:main}
than in (\ref{KCauchy}), since the quantity $c_1$ is usually of modest size and
$2 c_2$ will be much less than the value in (\ref{KCauchy}).
If $c_2$ is significantly larger than $c_1$, then the expression for $K$ in 
Theorem \ref{thm:main} is approximately equal to $2 c_2$:
\[
K = c_2 + c_2 \sqrt{1 + \frac{c_1}{c_2^2}} = 2 c_2 + \frac{1}{2} \frac{c_1}{c_2} +
c_2~ O \left( \frac{c_1}{c_2^2} \right)^2 .
\] 

In other cases, $| \lambda_{min} ( \mu ( \zeta_0 , \zeta_0' , A)) |$ may be
approximately equal to $\frac{1}{\pi}$ times the numerical radius of the resolvent 
$( \zeta_0 I - A )^{-1}$.
Since the numerical radius is between $\frac{1}{2}$ and $1$ times the resolvent norm, there may be little difference between the $K$ value in Theorem
\ref{thm:main} and that in (\ref{KCauchy}).  In fact, the value in (\ref{KCauchy})
may actually be smaller because it involves $\frac{1}{2 \pi}$ times the integral
of the resolvent norm, while $c_2$ in Theorem \ref{thm:main} involves 
the integral of $| \lambda_{min} ( \mu ( \zeta (s) , A ) ) |$, 
and $K$ in Theorem \ref{thm:main} is approximately $2 c_2$.
If $| \lambda_{min} ( \mu ( \zeta (s) , A )) | = \frac{1}{\pi} w( ( \zeta (s) I - A )^{-1} )$
and $w( ( \zeta (s) I - A )^{-1} ) = \| ( \zeta (s) I - A )^{-1} \|$, then the $K$
value in Theorem \ref{thm:main} could exceed that in (\ref{KCauchy}) by
a factor of $4$, plus a term involving $c_1$, but this is the most by which
the $K$ value in Theorem \ref{thm:main} can exceed that in (\ref{KCauchy}).

We will see in Section \ref{applications} that in many problems of interest -- problems
in which the matrix $A$ is highly nonnormal and a point $\zeta_0$ on the boundary of $\Omega$ 
comes close to some ill-conditioned eigenvalues of $A$ -- we do, indeed, find that 
$| \lambda_{min} ( \zeta_0 , \zeta_{0}', A ) | \approx \frac{1}{\pi} w( ( \zeta_0 I - A )^{-1} )$, 
and the bound on $K$ in (\ref{KCauchy}) is actually somewhat smaller than that in Theorem \ref{thm:main}.  
We do not yet have a complete explanation of this phenomenon, but here we give an indication of why this 
might be expected.
 
\subsection{When the Numerical Range of the Resolvent is Close to A Disk about a Point
Near the Origin}

First note that if $x$ and $y$ are two unit vectors that are orthogonal to each other, then the
numerical range of the rank one matrix $x y^{*}$ is a disk about the origin
of radius $\frac{1}{2}$.  To see this, consider a unitary similarity transformation
$Q^{*} x y^{*} Q$, where the columns of $Q$ are $[x, y, q_3 , \ldots , q_n ]$.
The matrix $Q^{*} x y^{*} Q$ is the direct sum of a $2$ by $2$ Jordan block
with eigenvalue $0$ and an $n-2$ by $n-2$ block of zeros; the numerical range 
of this matrix is a disk about the origin of radius $\frac{1}{2}$.
Note also that the 2-norm of this matrix is $1$, which is twice the numerical
radius.

If $x$ and $y$ are normalized right and left eigenvectors of $A$ 
corresponding to a simple eigenvalue $\lambda$, (i.e., $x$ and $y$ satisfy 
$A x = \lambda x$ and $y^{*} A = \lambda y^{*}$), the {\em condition number} of $\lambda$ 
is defined as $1/ | y^{*} x |$.  If $\lambda$ is ill-conditioned, then $y$ is almost
orthogonal to $x$.  The following theorem modifies the argument in the previous 
paragraph to deal with the case where $x$ and $y$ are {\em almost} orthogonal 
to each other.

\begin{theorem}
\label{thm:rankone}
Let $x$ and $y$ be unit vectors.  Then the rank one matrix $x y^{*}$
is unitarily similar to the direct sum of a certain $2$ by $2$ matrix
and an $n-2$ by $n-2$ block of zeros.  The $2$ by $2$ matrix is
\begin{equation}
\left[ \begin{array}{cc} \frac{1}{2} ( y^{*} x ) & 1 \\ 0 & \frac{1}{2} ( y^{*} x )
\end{array} \right] + E , \label{pertrankone}
\end{equation}
where the entries of $E$ have magnitude $O( | y^{*} x |^2 )$.
The numerical range of the first matrix in (\ref{pertrankone}) is a disk of radius
$\frac{1}{2}$ about $\frac{1}{2} ( y^{*} x )$, and its norm is $1 + O( | y^{*} x |^2 )$.
\end{theorem}
\begin{proof}
Let 
\begin{eqnarray*}
q_1 & = & \left( x - \frac{1}{2} ( y^{*} x ) y \right) / 
\left\| x - \frac{1}{2} ( y^{*} x) y \right\| , \\
\tilde{q}_2 & = & \left( y - \frac{1}{2} ( x^{*} y ) x \right) / 
\left\| y - \frac{1}{2} ( x^{*} y ) x \right\| , \\
q_2 & = & ( \tilde{q}_2 - ( q_1^{*} \tilde{q}_2 ) q_1 ) / 
\| \tilde{q}_2 - ( q_1^{*} \tilde{q}_2 ) q_1 \| ,
\end{eqnarray*}
and let $q_3 , \ldots , q_n$ be any orthonormal vectors that are orthogonal to
$q_1$ and $q_2$ (and hence to $x$ and $y$).  Note that 
\[
\tilde{q}_2^{*} q_1 = \frac{\frac{1}{4} ( y^{*} x ) | y^{*} x |^2}{1 - \frac{3}{4}
| y^{*} x |^2} ,
\]
so that $q_2^{*} x$ and $y^{*} q_2$ differ from $\tilde{q}_2^{*} x$
and $y^{*} \tilde{q}_2$ by at most terms of order $| y^{*} x |^3$.
Let $Q$ be the unitary matrix with columns $[ q_1 , \ldots , q_n]$.
Then $Q^{*} x y^{*} Q$ is the direct sum of a $2$ by $2$ matrix and an
$n-2$ by $n-2$ block of zeros, where the $2$ by $2$ matrix is
\[
\left[ \begin{array}{c} q_1^{*} \\ q_2^{*} \end{array} \right] x y^{*}
[ q_1 , q_2 ] =
\left[ \begin{array}{cc} ( q_1^{*} x ) ( y^{*} q_1 ) & ( q_1^{*} x ) ( y^{*} q_2 ) \\
( q_2^{*} x ) ( y^{*} q_1 ) & ( q_2^{*} x ) ( y^{*} q_2 ) \end{array} \right] =
\left[ \begin{array}{cc} \frac{1}{2} ( y^{*} x ) & 1 \\
0 & \frac{1}{2} ( y^{*} x ) \end{array} \right] + E ,
\]
where a straightforward calculation shows that each entry of $E$ is of order
$| y^{*} x |^2$.
\end{proof}

Assuming that $| y^{*} x | << 1$, Theorem \ref{thm:rankone} shows that the 
numerical range of the rank one matrix $x y^{*}$ is close to a disk, not about
the origin, but about a point $\frac{1}{2} ( y^{*} x )$ whose absolute value
is much less than the radius of the disk.  Hence each point on the
boundary of the numerical range has absolute value close to the numerical
radius.

Suppose $A$ is diagonalizable with eigenvalues $\lambda_1 , \ldots , \lambda_n$
and normalized right and left eigenvectors $x_1 , \ldots , x_n$ and
$y_1 , \ldots , y_n$.  Then the resolvent $( \zeta I - A )^{-1}$ can be
written in the form:
\[
( \zeta I - A )^{-1} = \sum_{j=1}^n \frac{1}{\zeta - \lambda_j} \frac{x_j y_j^{*}}
{y_j^{*} x_j} .
\]
If $\zeta$ is {\em much} closer to one eigenvalue, say, $\lambda_1$ than it is
to any of the others, then the first term above will be the largest, and
\begin{equation}
( \zeta I - A )^{-1} \approx \frac{1}{\zeta - \lambda_1} \frac{x_1 y_1^{*}}
{y_1^{*} x_1} . \label{x1y1p}
\end{equation}
If $\lambda_1$ is ill-conditioned so that $| y_1^{*} x_1 | << 1$, then
from Theorem \ref{thm:rankone}, the numerical range of $( \zeta I - A )^{-1}$
will be approximately equal to $1/ ( ( \zeta - \lambda ) ( y_1^{*} x_1 ) )$
times a disk of radius $\frac{1}{2}$ about the point $\frac{1}{2} ( y_1^{*} x_1 )$.
Thus each point on the boundary of the numerical range of the resolvent
will have absolute value approximately equal to the numerical radius of the
resolvent. 

In some of the examples of section \ref{applications}, we will encounter points $\zeta$
that are only {\em fairly} close to an ill-conditioned eigenvalue or are fairly close to
several ill-conditioned eigenvalues.  In this case, the approximate equality (\ref{x1y1p})
may not hold because other nearby eigenvalues still have an effect.
The closest (in 2-norm or Frobenius norm) rank one matrix to $( \zeta I-A )^{-1}$ is 
$\sigma_1 u_1 v_1^{*}$, where $\sigma_1$ is the largest {\em singular value} of
$( \zeta I - A )^{-1}$ and 
$u_1$ and $v_1$ are the associated left and right singular vectors, respectively.  
In this case, if $u_1$ and $v_1$ are almost orthogonal to each other, then
Theorem \ref{thm:rankone} shows that if $( \zeta I - A )^{-1} \approx \sigma_1 u_1 v_1^{*}$,
then the numerical range of $( \zeta I - A )^{-1}$ is approximately 
equal to $\sigma_1$ times a disk of radius $\frac{1}{2}$ about $\frac{1}{2} v_1^{*} u_1$.
Again, the radius is much larger than the absolute value of the center, so all
points on the boundary of this disk have absolute value close to the numerical radius.

To see that the right and left singular vectors corresponding to the largest
singular value of $( \zeta I - A )^{-1}$ are almost orthogonal to each other
when $\zeta$ is close to a simple but ill-conditioned eigenvalue $\lambda$ of $A$,
we can use a theorem of G.~W.~Stewart \cite{Stewart1973}.  First note that these
are the left and right singular vectors corresponding to the {\em smallest} singular
value of $\zeta I - A$.  Let us start with the matrix $\lambda I - A$, which
has a null space of dimension one.  The normalized right and left eigenvectors,
$x$ and $y$, corresponding to the eigenvalue $\lambda$ of $A$ satisfy
$( \lambda I - A ) x = 0$ and $( \lambda I - A )^{*} y = 0$.  It follows that
these are right and left singular vectors of $\lambda I - A$ corresponding to
the smallest singular value, $0$.  Write the SVD of $\lambda I - A$ as
$Y \Sigma X^{*}$, where $X = [x, X_2 ]$ and $Y = [y, Y_2]$, and we have put the
smallest singular value first.
Define $E := ( \zeta - \lambda ) I$ so that $( \lambda I - A)+E =
\zeta I - A$.  Define
\[
\gamma := \left\| \left[ \begin{array}{c} Y_2^{*} E x \\ X_2^{*} E^{*} y \end{array}
\right] \right\|_F = \left\| \left[ \begin{array}{c}
( \zeta - \lambda ) Y_2^{*} x \\ ( \bar{\zeta} - \bar{\lambda} ) X_2^{*} y \end{array}
\right] \right\|_F \leq \sqrt{2}~| \zeta - \lambda | ,
\]
\begin{eqnarray*}
\delta & := & \sigma_{n-1} ( \lambda I - A ) - \| y^{*} E x \|_2 - \| Y_2^{*} E X_2 \|_2 \\
 & = & \sigma_{n-1} ( \lambda I - A ) - | \zeta - \lambda | \left( | y^{*} x | +
\| Y_2^{*} X_2 \|_2 \right) \\
 & \geq & \sigma_{n-1} ( \lambda I - A ) - | \zeta - \lambda | ( 1 + | y^{*} x | ) ,
\end{eqnarray*}
where $\sigma_{n-1} ( \lambda I - A )$ is the second smallest singular value of
$\lambda I - A$.
Assuming that $\gamma / \delta < 1/2$, it is shown in \cite[Theorem 6.4]{Stewart1973}
that there are vectors $p$ and $q$ satisfying
\[
\left\| \left[ \begin{array}{c} p \\ q \end{array} \right] \right\|_F < 2 \frac{\gamma}
{\delta}
\]
such that $x + X_2 p$ and $y + Y_2 q$ are (multiples of) right and left singular
vectors of $(\lambda I - A )+E = \zeta I - A$, corresponding to the smallest singular value;
i.e., they are left and right singular vectors of $( \zeta I - A )^{-1}$, corresponding
to the largest singular value.  It follows that if $x$ and $y$ are almost orthogonal
to each other and if $\| p \|_2$ and $\| q \|_2$ are small, then the singular vectors
$u_1$ and $v_1$ corresponding to the largest singular value of $( \zeta I - A )^{-1}$
are almost orthogonal to each other:
\[
\left|
\frac{( x + X_2 p )^{*} ( y + Y_2 q )}{\| x + X_2 p \|_2 \| y + Y_2 q \|_2} \right| =
\frac{| x^{*} y + x^{*} Y_2 q + p^{*} X_2^{*} y + p^{*} X_2^{*} Y_2 q |}
{\| x + X_2 p \|_2 \| y + Y_2 q \|_2} \leq
\frac{| x^{*} y | + \| q \|_2 + \| p \|_2 + \| p \|_2 \| q \|_2}
{\sqrt{( 1 - \| p \|_2^2 ) ( 1 - \| q \|_2^2 )}} .
\]

To get an idea of why $( \zeta I - A )^{-1}$ looks like a rank one matrix when
$\zeta$ is close to an ill-conditioned eigenvalue $\lambda$ of $A$, we will use  
a theorem of M.~Stewart \cite{MStewart2006}.  While typically one expects the singular
values of $\zeta I - A$ to differ from those of $\lambda I - A$ by 
$O( | \zeta - \lambda | )$ (see, for instance, \cite[Theorem 3.3.16]{HJ2}), 
Stewart showed that the smallest singular value changes from $0$ to only
\[
| \zeta - \lambda |~| y^{*} x | + O( | \zeta - \lambda |^2 ) ,
\]
where $y$ and $x$ are the left and right singular vectors of $\lambda I - A$
associated with the zero singular value.  (He also described the $O( | \zeta -
\lambda |^2 )$ terms.)  The second smallest singular value $\sigma_{n-1}$
decreases by at most $| \zeta - \lambda |$ and might increase by this amount,
so as long as $| \zeta - \lambda | << \sigma_{n-1}$, we can expect the ratio
of smallest to second smallest singular value of $\zeta I - A$ (i.e.,
the ratio of second largest to largest singular value of the resolvent
$( \zeta I - A )^{-1}$) to remain small.

\begin{figure}[t]
    \centering
    \includegraphics[height=3in]{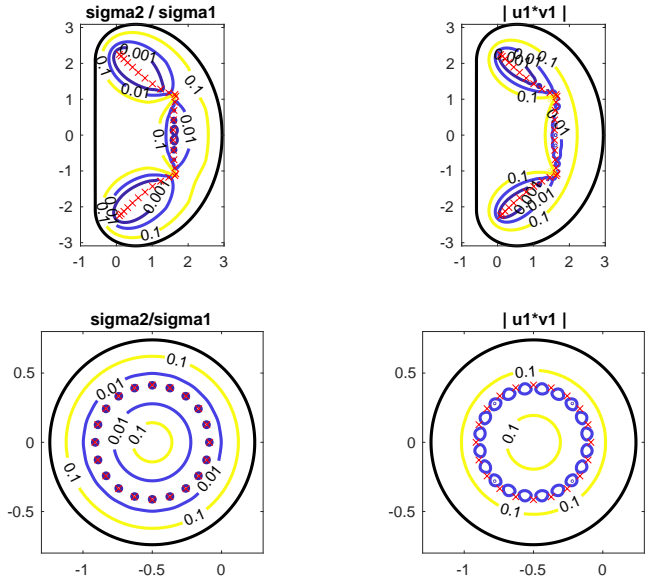}
    \caption{Contour plots of ratios of second largest to largest singular value
of $( \zeta I - A )^{-1}$ and of inner products $| u_1^{*} v_1 |$ of left and right
singular vectors corresponding to the largest singular value of $( \zeta I - A )^{-1}$
for the Grcar matrix of size $n=32$ (top) and the \texttt{transient\_demo} matrix of
size $n=20$ (bottom).  Also shown are the eigenvalues ($x$) and the boundary of the
numerical range (thick black curve).}
    \label{fig:rankone}
\end{figure}

To illustrate this phenomenon, Figure \ref{fig:rankone} shows contour plots of 
the ratios of second largest to largest singular value of $( \zeta I - A )^{-1}$ 
and of the inner products $| u_1^{*} v_1 |$ of the
left and right singular vectors corresponding to the largest singular value
of $( \zeta I  - A )^{-1}$ for two highly nonnormal matrices.  Note the large
areas over which these ratios and inner products are small, implying that the numerical
range of the resolvent is close to a disk about a point much nearer to the origin
than the radius of the disk.  

\begin{figure}[t!]
    \centering
    \includegraphics[height=2in]{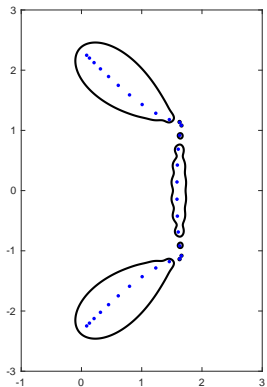}
    \caption{Matrix from MATLAB command \texttt{gallery('grcar',32)}
Eigenvalues (dots) and components of the $10^{-3}$-pseudospectrum (solid curves).
Direct application of Theorem \ref{thm:main} shows that this is a
$4.20 \times 10^3$-spectral set for $A$, but the value of $K$ from (\ref{KCauchy}) is
$2.12 \times 10^3$.}
    \label{fig:pseudo}
\end{figure}

The top plots are for the Grcar matrix of size $n = 32$.
This matrix has $-1$'s on the subdiagonal, $1$'s on the main diagonal and the first 
three super-diagonals, and $0$'s elsewhere.  
It was shown in \cite{CG} (for a Grcar matrix of size $100$)
that the $K$ value obtained from Theorem \ref{thm:main} is much smaller than that from
(\ref{KCauchy}) {\em if} the region $\Omega$ is taken to be 
$W(A) \backslash {\cal D}(0,1/ w( A^{-1} ))$.  Figure \ref{fig:rankone} shows that this
will not be the case if one chooses a smaller region $\Omega$; e.g., the $10^{-3}$ 
pseudospectrum, pictured in Figure \ref{fig:pseudo}.  This region looks similar
to the $0.01$ level curve of $\sigma_2 / \sigma_1$, so at points on the boundary of 
the $10^{-3}$ pseudospectrum, the resolvent $( \zeta I - A )^{-1}$ is close to a 
rank one matrix.
The bottom plots in Figure \ref{fig:rankone} are for the \verb+transient_demo+
matrix of size $20$, available in the eigtool package \cite{eigtool},
which will be used in section \ref{applications}.

\section{Applications} \label{applications}
Throughout this section and the next, we will always assume that the space $H$
in which we are working is Euclidean space and the norm of interest is the 2-norm,
which will be denoted as $\| \cdot \|_2$.
MATLAB codes used to produce the results in this section can be found at:
\verb+http://tygris/k-spectral-sets+.

\subsection{Block Diagonal Matrices} \label{blockd}
If $A$ is a block diagonal matrix, say,
\[
A = \left[ \begin{array}{cc} A_{11} & 0 \\ 0 & A_{22} \end{array} \right] ,
\]
then since
\[
f(A) = \left[ \begin{array}{cc} f( A_{11} ) & 0 \\ 0 & f( A_{22} ) \end{array} \right] ,
\]
it is clear that $\| f(A) \|_2$ can be bounded based on the size of $f$ on 
$W( A_{11} ) \cup W( A_{22} )$.  Yet $W(A)$ is a possibly larger set:  
the convex hull of $W( A_{11} ) \cup W( A_{22} )$.  Of course, if one
knew that $A$ was block diagonal, then one could take advantage of this property,
but the same observation holds when $A$ is unitarily similar to a block diagonal matrix,
and then it is an np-hard problem to identify the blocks \cite{Gu1995}.   Instead, one might
start with $W(A)$ and try to remove one or more disks that would cut the region 
into disjoint pieces corresponding to the blocks of $A$.  

An example is illustrated 
in Figure \ref{fig:blockdiag}.  For this matrix, $A_{11}$ was a real random 
$4$ by $4$ matrix and $A_{22}$ was equal to $8I$ plus a real random
$4$ by $4$ matrix, where the random matrix entries were drawn from a standard normal distribution.
The disk removed was centered at $\xi = 3.5$ and had radius 
$1/w( ( \xi I - A )^{-1} )$.  According to the results of section \ref{example}, 
the remaining region (outlined with a thick black line in the figure)
is a $(3 + 2 \sqrt{3})$-spectral set for $A$.  For comparison, if one evaluates
the resolvent norm integral in (\ref{KCauchy}) over the boundary of this set, one obtains
the slightly larger value of $8.01$.  Also shown in red in the figure are the
numerical ranges of each block.

\begin{figure}[ht]
\centering
\includegraphics[width = 3 in]{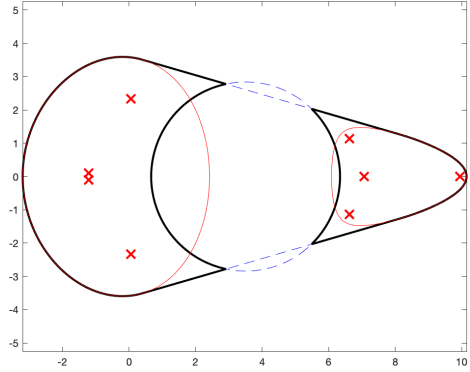}
\caption{Eigenvalues and numerical range of a block diagonal matrix cut into two 
pieces by removing a disk about $\xi = 3.5$ of radius $1/w ( ( A - \xi I )^{-1} )$.  
Resulting region is outlined in black; numerical ranges of the blocks
are shown in red.
}
\label{fig:blockdiag}
\end{figure}

For a matrix with more diagonal blocks, one could remove more disks from $W(A)$
and obtain a $K$-spectral set with three or more disjoint simply connected regions, 
where $K$ is bounded by expression \eqref{mdisks}.
In other cases, a single disk may not be wide enough to split the numerical
range into disjoint pieces.  Then multiple disks could be removed, and $K$ would again
be bounded by expression (\ref{mdisks}).
A better bound might be obtained by using Theorem \ref{thm:main} directly and 
numerically determining bounds on $c_1$ and $c_2$, as described in section \ref{OtherKSpectral}.  
Figures \ref{fig:block1} and \ref{fig:block0} show additional illustrations, 
along with the $K$ value obtained from formula (\ref{mdisks}) and one computed directly from 
Theorem \ref{thm:main}.

\begin{figure}[h!]
    \centering
    \includegraphics[width = 3in]{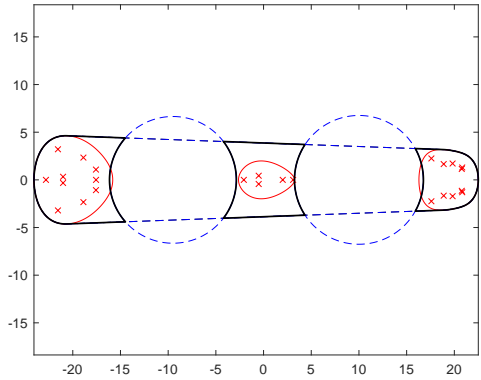}
\caption{$A$ is a block diagonal matrix with three blocks.
Each block is the sum of a multiple of the identity and a real random matrix $R$ 
with entries from a standard normal distribution.  Block $A_{11} = -20I + R_1$ is
$10$ by $10$, block $A_{22} = R_2$ is $5$ by $5$, and 
block $A_{33} = 20I + R_3$ is $10$ by $10$.
The disks removed had radii $1/ \| ( \xi_{1,2} I - A )^{-1} \|_2$, where $\xi_1 = -9.5$
and $\xi_2 = 10$.  Based on formula (\ref{mdisks}),
the remaining region is a $K = 3 + \sqrt{14} \approx 6.74$ spectral set, and using
Theorem \ref{thm:main} directly, as described in section \ref{OtherKSpectral}, we computed
$c_1 \leq 2.60$, $c_2 \leq 1.78$, and $K = 4.19$.  Using formula
\ref{KCauchy}, the value of $K$ was computed to be $11.88$.}
    \label{fig:block1}
\end{figure}
\begin{figure}[h!]
    \centering
    \includegraphics[width = 3 in]{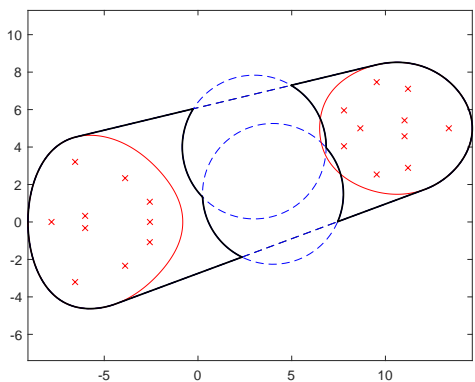}
    \caption{$A$ is a block diagonal matrix with two blocks.
Each block is the sum of a multiple of the identity and a real random matrix $R$ 
with entries from a standard normal distribution.  Block $A_{11} = -5I + R_1$ 
is $10$ by $10$, and block $A_{22} = (10+5i)I + R_2$ is $10$ by $10$.  Two disks of radius
$1/ \| ( \xi_{1,2} I - A )^{-1} \|_2$, where $\xi_1 = 4 + 1.5i$ and $\xi_2 = 3+4i$,
were needed to split the numerical range of $A$
into two disjoint sets.  Based on formula (\ref{mdisks}), the remaining
region is a $3 + \sqrt{14} \approx 6.74$ spectral set, and using Theorem
\ref{thm:main} directly, as described in section \ref{OtherKSpectral}, we
computed $c_1 \leq 3.20$, $c_2 \leq 1.73$, and $K = 4.21$.  Using formula
\ref{KCauchy}, the value of $K$ was computed to be $7.94$.}
    \label{fig:block0}
\end{figure}

\newpage
\subsection{Bounding Solutions to the Initial Value Problem} \label{ds}

The results from section \ref{OtherKSpectral} can be used to bound the solutions 
to both continuous and discrete time dynamical systems, assuming that the 
spectrum of $A$ lies in the left half-plane or the unit disk, respectively, 
by determining a $K$ value for the set $\Omega$ equal to the intersection of $W(A)$ 
with the left half-plane or the unit disk.

In this case, since $\Omega$ is simply connected, one {\em may} be able to find
the {\em optimal} $K$ value numerically.  If $A$ is an $n$ by $n$ matrix,
then the form of the function $f$ with $\| f \|_{\Omega} = 1$ that maximizes 
$\| f(A) \|$ is known; it is of the form $B \circ \varphi$, where $\varphi$ is any conformal 
mapping from $\Omega$ to the unit disk and $B$ is a finite Blaschke product of degree 
at most $n-1$:
\[
B (z) = \prod_{j=1}^{n-1} \frac{z - \alpha_j}{1 - \bar{\alpha}_j z} ,~~
| \alpha_j | \leq 1 .
\]
We use the Kerzmann-Stein procedure \cite{KS,KT} as implemented in \verb+chebfun+ \cite{chebfun}
to conformally map $\Omega$ to the unit disk.  We then try many different initial guesses
for the roots $\alpha_j$ of $B$ and use the optimization code \verb+fmincon+ in MATLAB
to search for roots that maximize $\| B( \varphi (A)) \|_2$.  We can check a number of 
conditions
that are known to hold for the optimal Blaschke product $B$ to give us some confidence
that we  have indeed found the global maximum.  See \cite{BGGRSW} for details.  Still,
these conditions are not sufficient to guarantee a global maximum, but at least
the maximum value of $\| B( \varphi (A) ) \|_2$ returned by the optimization code is a 
{\em lower bound} on the optimal $K$ value for the region $\Omega$.

As an example, the left plot in Figure \ref{fig:tuesday} shows 
the behavior of $\| e^{tA} \|_2$ for a matrix $A$ from \cite{NC} that
models the ecosystem of Tuesday Lake in Wisconsin after the introduction of 
pisciverous largemouth bass.  The plot shows initial growth
and then decay of the relative total population of the Tuesday Lake ecosystem.  The right plot
in the figure shows the eigenvalues and numerical range of the matrix
and the part of the numerical range in the left half-plane.  In this
case we found, by integrating $| \lambda_{min} ( \mu ( \zeta (s),A ) |$
along the segment of the imaginary axis inside $W(A)$ and using
Theorem \ref{thm:main}, that $K$ could be bounded by $2.66$, while
formula (\ref{KCauchy}) gave the slightly larger value $K = 3.72$.
Based on results from our optimization code, we believe that the optimal value
of $K$ for this region is $1.95$, and, as noted earlier, this is at least 
a lower bound on $K$.  In this case the different bounds on $K$ are all very close
and somewhat larger than the maximum value of $\| e^{tA} \|_2$, $t > 0$,
found in Figure \ref{fig:tuesday}.

\begin{figure}[t!]
\centerline{\epsfig{file=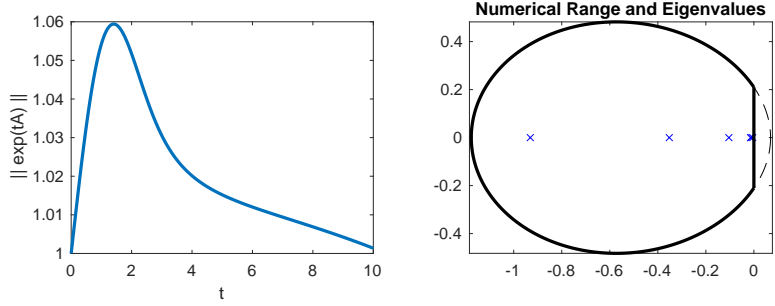,width=4in}}
\caption{Matrix modeling the ecosystem in Tuesday Lake after introducing
piscivores \cite{NC}.  Left plot shows $\| e^{tA} \|_2$ growing before 
decaying; right plot shows $W(A)$ extending into the right half-plane 
(dashed curve) and eigenvalues ($x$'s) in the left half-plane.} 
\label{fig:tuesday}
\end{figure}

\begin{figure}[t!]
\centerline{\epsfig{file=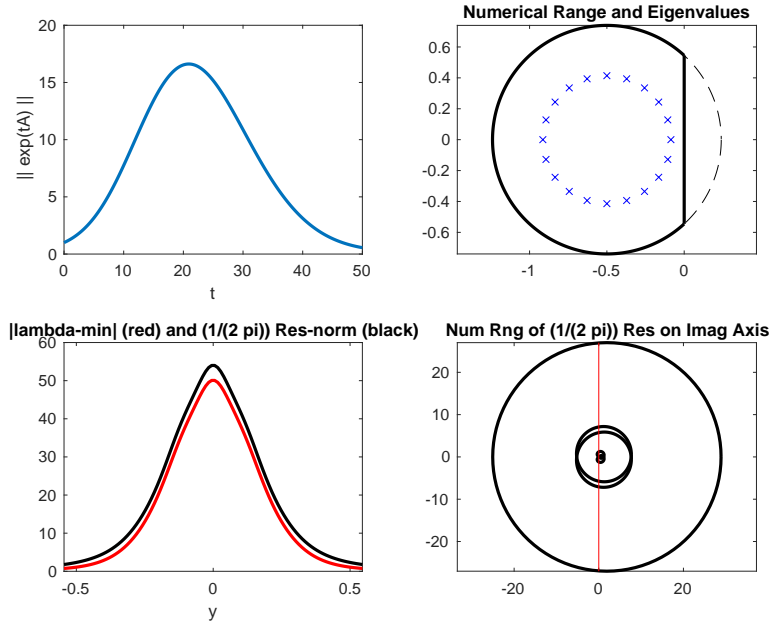,width=4in}}
\caption{Matrix from the eigtool command \texttt{transient\_demo(20)} 
\cite{eigtool}. Upper left shows $\| e^{tA} \|_2$ growing before decaying; 
upper right shows $W(A)$ extending into the right half-plane (dashed curve) 
and eigenvalues of $A$ (x's) in the left half-plane. 
Lower left shows $| \lambda_{\min}(\mu(\zeta),A)) |$ 
and $\frac{1}{2 \pi} \| ( \zeta I - A )^{-1} \|_2$ for $\zeta$ 
on the segment of the imaginary axis forming the right boundary of 
$\Omega$.  Lower right shows numerical ranges of several of the matrices
$\frac{1}{2 \pi} ( \zeta I - A )^{-1}$ for $\zeta$ on this segment of the
imaginary axis; for the larger numerical ranges, the absolute value of the
minimal real part, which is 
$\frac{1}{2} | \lambda_{min}( \mu ( \zeta , A ) ) |$, is almost as large
as the numerical radius, explaining why $| \lambda_{min} ( \mu ( \zeta , A ) ) |$
is of the same order of magnitude as $\frac{1}{2 \pi} \|(\zeta I-A)^{-1}\|_2$.} 
\label{fig:transient20}
\end{figure}

As another example, we consider the matrix 
\texttt{transient\_demo(20)} available in the eigtool
package \cite{eigtool}.  The upper left plot in Figure \ref{fig:transient20}
shows the behavior of $\| e^{tA} \|_2$, $t > 0$, which grows to about
$16.61$ before starting to decrease.  The upper right plot shows the 
eigenvalues, in the left half-plane, and the numerical range, extending
into the right half-plane, together with the region $\Omega$ consisting
of the part of $W(A)$ in the left half-plane.   
Integrating $| \lambda_{min} ( \mu ( \zeta (s),A )) |$ along the segment
of the imaginary axis forming the right boundary of $\Omega$ and
using Theorem \ref{thm:main}, we determined that 
$K = c_2 + \sqrt{ c_2^2 + c_1 } \approx 2 c_2 = 40.13$.  In this case,
formula (\ref{KCauchy}) gave a smaller value, $K = 27.95$.  The
reason for this smaller value can be seen in the lower plots of 
Figure \ref{fig:transient20}.  The large values of 
$| \lambda_{min} ( \mu ( \zeta (s),A ) ) |$ and of 
$ \frac{1}{2 \pi} \| ( \zeta I - A )^{-1} \|_2$
occur on the segment of the imaginary axis, and, while
$| \lambda_{min} ( \mu ( \zeta (s),A ) |$ is always less than 
or equal to $\frac{1}{2 \pi} \| ( \zeta I - A )^{-1} \|_2$, the difference
is small.  Since the value of $K$ from Theorem \ref{thm:main}
is approximately equal to $2 c_2$, which is approximately twice
the integral of $| \lambda_{min} ( \mu ( \zeta (s),A ) ) |$ over this 
segment, and the value of $K$ from (\ref{KCauchy}) is the integral 
of $\frac{1}{2 \pi} \| ( \zeta I - A )^{-1} \|_2$ over
this segment (and over the remainder of $\partial \Omega$, where 
$\| ( \zeta I - A )^{-1} \|_2$ is much smaller), the result is a 
smaller value of $K$ from formula (\ref{KCauchy}). 
The lower right plot shows why $| \lambda_{min} ( \mu ( \zeta ,A )) |$ might
be almost as large as $\frac{1}{2 \pi} \| ( \zeta I - A )^{-1} \|_2$.  It shows
the numerical ranges of several of the matrices $\frac{\zeta'}{2 \pi i} 
( \zeta I - A )^{-1} = \frac{1}{2 \pi} ( \zeta I - A )^{-1}$
for $\zeta$ on this segment of the imaginary axis.  While the smaller numerical
ranges lie mostly in the right half-plane, for the larger ones, the absolute value of the
real part of the leftmost point in these numerical ranges (which is $\frac{1}{2} | \lambda_{min} (
\mu ( \zeta , A ) ) |$) is almost as large as the numerical radius.
We will later see why this might be expected when $\zeta$ is close to
an ill-conditioned eigenvalue.
In this example, our optimization code found a function $B \circ \varphi$
for which $\| B( \varphi (A) ) \|_2 = 21.54$, and we believe that this
is the optimal value of $K$ for this set $\Omega$.

\begin{figure}[t]
    \centering
    \includegraphics[width=4in]{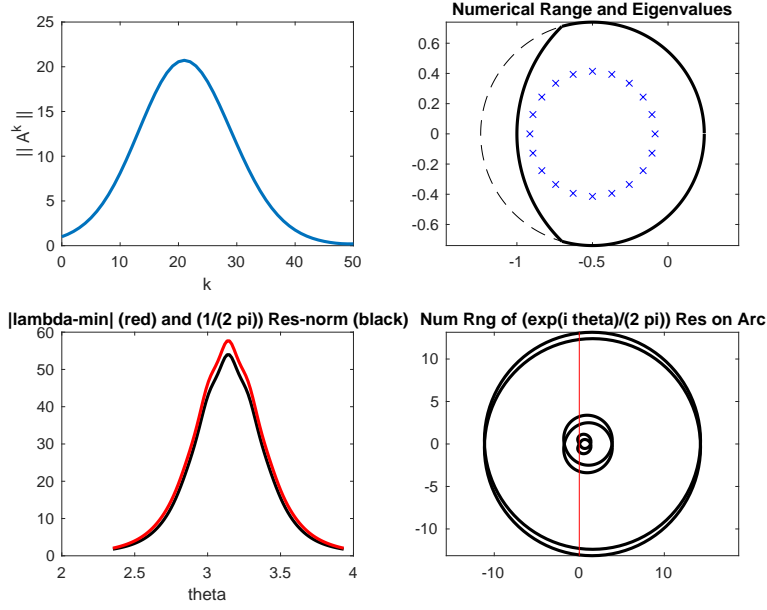}
    \caption{Matrix from the eigtool command 
\texttt{transient\_demo(20)} \cite{eigtool}. Upper left shows $\|A^k\|_2$ 
growing before decaying; upper right shows $W(A)$ extending beyond 
$\mathcal D(0,1)$ (dashed curve) and eigenvalues of $A$ (x's) in the unit disk. 
Lower left shows $| \lambda_{\min}(\mu(\zeta ,A)) |$ and
$\frac{1}{2 \pi} \| ( \zeta I - A )^{-1} \|_2$ for $\zeta$ on the arc
of the unit circle inside $W(A)$. 
Lower right shows numerical ranges of several of the matrices
$\frac{e^{i \theta}}{2 \pi} ( \zeta I - A )^{-1}$ for $\zeta = e^{i \theta}$ on this arc of the
unit circle; for the larger numerical ranges, the absolute value of the
minimal real part, which is $\frac{1}{2} | \lambda_{min}( \mu ( \zeta , A ) ) |$, 
is almost as large as the numerical radius, explaining why 
$| \lambda_{min} ( \mu ( \zeta , A ) ) |$
is of the same order of magnitude as $\frac{1}{2 \pi} \|(\zeta I-A)^{-1}\|_2$.} 
    \label{fig:transientpowers}
\end{figure}

Using the same matrix, \texttt{transient\_demo(20)}, 
we computed norms of powers of $A$ and found that they grew to about
$20.72$ before starting to decrease, as shown in the upper left plot
of Figure \ref{fig:transientpowers}.
The upper right plot shows the numerical range 
of the matrix, which extends beyond $\mathcal D(0,1)$, and the 
eigenvalues which all lie within $\mathcal D(0,1)$. 
If we take $\Omega$ to be $W(A) \cap \mathcal D(0,1)$, 
whose boundary is the wide solid line in the upper-right plot, 
then we can use Theorem \ref{thm:main} to calculate a value of $K$ 
for which $\Omega$ is a $K$-spectral set.
Integrating $| \lambda_{min} ( \mu ( \zeta (s),A )) |$ along the arc
of the unit circle inside $W(A)$, we determined that 
$K = c_2 + \sqrt{ c_2^2 + c_1 } = 70.44$.  Again in this case,
formula (\ref{KCauchy}) gave a smaller value, $K = 36.03$.  The
reason can be seen in the lower plots of 
Figure \ref{fig:transientpowers}.  The large values of 
$| \lambda_{min} ( \mu ( \zeta ,A ) ) |$ and of 
$ \frac{1}{2 \pi} \| ( \zeta I - A )^{-1} \|_2$
occur on the arc of the unit circle inside $W(A)$, as shown in the 
lower left plot.  In this case,
$| \lambda_{min} ( \mu ( \zeta (s),A ) |$ is greater than 
$\frac{1}{2 \pi} \| ( \zeta I - A )^{-1} \|_2$.
The lower right plot shows why $| \lambda_{min} ( \mu ( \zeta ,A )) |$ might
be larger than $\frac{1}{2 \pi} \| ( \zeta I - A )^{-1} \|_2$.  It shows
the numerical ranges of several of the matrices $\frac{\zeta'}{2 \pi i}
( \zeta I - A )^{-1} = \frac{e^{i \theta}}{2 \pi} ( \zeta I - A )^{-1}$
for $\zeta = e^{i \theta}$ on this arc of the unit circle.  For the larger numerical ranges,
the absolute value of the real part of the
leftmost point in these numerical ranges (which is $\frac{1}{2} | \lambda_{min} (
\mu ( \zeta , A ) ) |$) is almost as large as the numerical radius.
Again, we will give a partial explanation for this in the last section.
In this example, our optimization code found a function $B \circ \varphi$
for which $\| B( \varphi (A) ) \|_2 = 21.06$, and we believe that this
is the optimal value of $K$ for this set $\Omega$.

\section{Summary and Concluding Remarks} \label{comparisons}
The examples of the previous section show that for certain sets $\Omega$,
Theorem \ref{thm:main} provides smaller $K$ values than (\ref{KCauchy}),
but for other sets $\Omega$, this is not the case.  Figures
\ref{fig:transient20} and \ref{fig:transientpowers} show that when the 
$K$ value from (\ref{KCauchy}) is smaller, it is because, for points $\zeta$
on $\partial \Omega$ where $\frac{1}{2 \pi} \| ( \zeta I - A )^{-1} \|_2$ is large, 
the quantity $| \lambda_{min} ( \mu ( \zeta , A ) ) |$ is
about the same size.  This is because $| \lambda_{min} ( \mu ( \zeta , A ) ) |$
is the absolute value of a particular point on the boundary of the numerical
range of $( \zeta I - A )^{-1}$, and the numerical range of $( \zeta I - A )^{-1}$
looks almost like a disk about the origin, or about a point near the origin. 
Thus $| \lambda_{min} ( \mu ( \zeta , A )) |$
is approximately equal to the numerical radius of $( \zeta I - A )^{-1}$, which
is within a factor of $2$ of the norm of $( \zeta I - A )^{-1}$.

In section \ref{relation}, we gave an explanation as to why this might be expected.
In areas near ill-conditioned eigenvalues, the resolvent looks like the rank
one matrix $\sigma_1 u_1 v_1^{*}$, where $\sigma_1$ is the largest singular value
of the resolvent and $u_1$ and $v_1$ are the corresponding left and right
singular vectors.  Additionally, $u_1$ and $v_1$ are almost orthogonal to each
other.  While Theorem \ref{thm:rankone} and the references thereafter
about perturbation of singular values and singular vectors give some insight
into where in the complex plane this phenomenon occurs, a more quantitative
analysis would be an interesting next step.  
Regions that come close to ill-conditioned eigenvalues are often the
most interesting ones for applications.

\end{document}